\documentclass[leqno,11pt]{amsart}
\usepackage{graphicx,pstricks,pst-plot}
\usepackage[latin1]{inputenc}
\usepackage[normalem]{ulem}
\usepackage{epsfig}
\usepackage{amsfonts}
\usepackage{graphicx}
\usepackage{graphicx,color}
\usepackage{footnote}
\usepackage{amsthm}
\usepackage{amsbsy}
\usepackage{amsmath}
\usepackage{amssymb}
\usepackage{rotating}
\usepackage[colorlinks=true,citecolor=blue,linkcolor=blue]{hyperref}
\usepackage[T1]{fontenc}
\hypersetup{hypertex,colorlinks=true,linkcolor=blue,filecolor=blue,pagecolor=blue,urlcolor=blue}

\usepackage{ragged2e}
\usepackage[
lmargin=3.5cm,
rmargin=3.5cm,
tmargin=4.0cm,
bmargin=4.0cm
]{geometry}

\theoremstyle{plain}
\newtheorem{theorem}{\indent\sc Theorem}[section]
\newtheorem{lemma}[theorem]{\indent\sc Lemma}

\theoremstyle{definition}

\makeatletter
\@namedef{subjclassname@2020}{%
  \textup{2020} Mathematics Subject Classification}
\makeatother

\numberwithin{equation}{section}

\title[On the rigidity of spacelike self-shrinker in $\mathbb{R}^{n+p}_p$]{Rigidity Results for Spacelike Self-Shrinkers \\ via Different Maximum Principles}

\author[W.F.C. Barboza]{Weiller F.Chaves Barboza$^{1}$}

\address{
$^1$ Weiller F. Chaves barboza \\ Departamento de Matem\'atica\\
Universidade Federal de Campina Grande\\
58.429-970 Campina Grande, Para\'iba\\
Brazil}
\email{weiller.felipe@professor.ufcg.edu.br}



\thanks{$^1$Corresponding author.}

\begin{document}

\maketitle

\begin{abstract}
In this work, we establish several rigidity results for spacelike self-shrinkers immersed in the pseudo-Euclidean space $\mathbb{R}^{n+p}_p$. Under suitable boundedness conditions on either the mean curvature vector or the second fundamental form, we apply different versions of Omori--Yau type maximum principles due to Qiu~\cite{Qiu:2021}, Chen and Qiu~\cite{QH:16}, and Alías, Caminha, and Nascimento~\cite{Alias-Caminha-Nascimento:2019} to show that such self-shrinkers must be spacelike hyperplanes. These results contribute to the broader classification of spacelike self-shrinkers under natural geometric assumptions.

\end{abstract}

\section{Introduction}\label{introduction}
\justifying

Let $X:M^n\rightarrow\mathbb R_p^{n+p}$ be a spacelike submanifold (which means that it has a Riemannian induced metric) in the $(n+p)$-dimensional pseudo-Euclidean space $\mathbb R^{n+p}_p$ of index $p$. The {\em spacelike mean curvature flow} (SMCF) associated to $X$ is a family of smooth spacelike submanifolds $X_t=X(t,\cdot):M^n\rightarrow\mathbb R_p^{n+p}$ with corresponding images $M^n_t=X_t(M^n)$ satisfying the following evolution equation
\begin{equation*}\label{eq-int:1}
	\left\{\begin{array}{l}\dfrac{\partial X}{\partial t}=\vec{H}\\ X(0,x)=X(x)
	\end{array}\right.
\end{equation*}
on some time interval, where $\vec{H}$ stands for the (non-normalized) mean curvature vector of the spacelike submanifold $M^n_t$ in $\mathbb R_p^{n+p}$. Self-similar shrinkers to the above SMCF play an important role in understanding the behavior of the flow since they often occur as singularities. $M^n$ is called a {\em spacelike self-shrinker} of $\mathbb{R}^{n+p}_{p}$ if it satisfies a system of quasilinear elliptic PDE of the second order
\begin{equation}\label{self}
	\vec{H}=-\frac{1}{2}X^{\perp},
\end{equation}
where $X^\perp$ denotes the normal part of $X$. The study of self-shrinkers in Euclidean spaces is deeply connected to minimal submanifold theory \cite{Angenent1992} and \cite{ColdingMinicozzi2012}. Extensive research has been conducted on classification and uniqueness problems for these geometric objects, including both self-shrinkers and translating solitons in Euclidean settings (See \cite{AbreschLanger1986}, \cite{CaoLi2013}, \cite{Wang2011}, \cite{Wang2011b}, \cite{NevesTian2013}, \cite{Qiu:2021}). Rigidity properties of complete spacelike submanifolds have attracted considerable attention in geometric analysis. The subject originated with Calabi's \cite{Calabi1968} seminal work on extremal hypersurfaces in $\mathbb{R}^{n+1}_1$, proving their planar nature in dimensions $n \leq 4$. A complete solution for all dimensions was later achieved by Cheng--Yau \cite{ChengYau1976}, contrasting sharply with Euclidean analogues. These results were ultimately extended to higher codimensional settings through the work of Jost--Xin \cite{JostXin2001}. 

The aforementioned considerations naturally motivate the study of rigidity problems for spacelike self-shrinkers, where several important results have been established: Chau--Chen--Yuan \cite{ChengYau1976} and Huang--Wang \cite{HuangWang2011} independently proved, through different methods, that spacelike entire graphical Lagrangian self-shrinkers must be flat under quadratic lower bounds on the Hessian of the potential function; Ding--Wang \cite{Ding-Wang:2010} subsequently obtained rigidity results under weaker subexponential decay conditions, while Ding--Xin \cite{DingXin2014} completely removed the auxiliary assumptions from \cite{ChauChenYuan2012,HuangWang2011}, definitively establishing flatness; additional classification and rigidity theorems under various geometric constraints were further developed in \cite{Adames2014,Liu-Xin:2016}.

In the current literature there are several works devoted to study the rigidity of spacelike self-shrinkers of $\mathbb{R}^{n+p}_{p}$. For instance, Ding and Wang~\cite{Ding-Wang:2010} investigated self-shrinking graphs with high codimension in $\mathbb{R}^{n+p}_{p}$ and obtained several Bernstein type results. Later on, Liu and Xin~\cite{Liu-Xin:2016} proved some rigidity results under minor growth conditions in terms of the mean curvature or the image of Gauss maps, and one of which is as follows.

\begin{theorem}[H.Q. Liu, Y.L. Xin \cite{Liu-Xin:2016}]
Let $M^n$ be a spacelike self-shrinker in $\mathbb{R}^{n+p}_p$, which is closed with respect to the Euclidean topology. If there exists a constant $\alpha < 1/8$ such that $||H||^2\leq e^{\alpha z}$, then $M^n$ is an affine $n-$plane.
\end{theorem}

In~\cite{QH:16}, Chen and Qiu proved that any complete $n$-dimensional spacelike self-shrinkers of $\mathbb{R}^{n+p}_{p}$ must be flat. Afterwards, Luo and Qiu~\cite{Luo-Qiu:2020} used an integral method to prove other rigidity theorem for spacelike self-shrinkers under a suitable constraint on the norm of the mean curvature vector, that is:

\begin{theorem}[Q. Chen, H.B. Qiu in \cite{QH:16}]
Let $X: M^n \longrightarrow \mathbb{R}^{n+p}_p$ be a $\xi-$submanifold which is closed	with respect to the Euclidean topology. Assume that the origin $0 \in M$ and there exists an $\varepsilon \in (0,1]$ and $\alpha < \varepsilon / 16$ such that 
$$
||H|| \leq C_ 0 ||B||^{1-\varepsilon} e^{\alpha z},
$$
where $C_0$ is a positive constant. Then $M^n$ must be an affine $n-$plane.
\end{theorem}

Motivated by the aforementioned works, we investigate the rigidity of spacelike self-shrinkers immersed in the $(n+p)-$dimensional pseudo-Euclidean space $\mathbb{R}^{n+p}_p$ with index $p$. Initially, assuming appropriate boundedness conditions on the mean curvature vector, we apply an Omori--Yau type maximum principle for the operator $\Delta_V$, due to Qiu~\cite{Qiu:2021}, to prove that spacelike hyperplanes are the only $n$-dimensional spacelike self-shrinkers immersed in $\mathbb{R}^{n+p}_p$. Subsequently, under suitable boundedness conditions on the second fundamental form, we employ another Omori--Yau type maximum principle, established by Chen and Qiu~\cite{QH:16}, to derive a similar rigidity result. Finally, by imposing a geometric condition on the second fundamental form and employing a maximum principle developed by Alías, Caminha, and Nascimento~\cite{Alias-Caminha-Nascimento:2019} which is based on the notion of functions converging to zero at infinity on complete noncompact Riemannian manifolds we obtain additional rigidity results concerning spacelike self-shrinkers in $\mathbb{R}^{n+p}_p$. Our main results are stated in the following three theorems:

\begin{theorem}\label{thm1}
Let $X:M^n\looparrowright\mathbb R^{n+p}_{p}$ be a spacelike self-shrinker, which is closed with respect to the Euclidean topology and assume that origin $0 \in M$. If there exists a constant $C>0$, such that 
$$||H||\leq C(d+1), \quad  d = \langle X, X \rangle$$ 
where $d$ is the pseudo-distance function, then $M^n$ is an $n$-dimensional spacelike hyperplane of $\mathbb R^{n+p}_{p}$.
\end{theorem}

\begin{theorem}\label{thm3}
	Let $X:M^n\looparrowright\mathbb R^{n+p}_{p}$ be a spacelike self-shrinker with  $||X^\perp||^2 < 2p$. Assume that the second fundamental form $A$ satisfies
	$$
	|A| \leq \dfrac{C}{r^{1+\varepsilon}}, \quad \text{and} \quad \rho_0 \geq \left( \dfrac{C_0}{\varepsilon} \right)^{1/\varepsilon},  ~~ C_0 = \sqrt{2p}C,
	$$
	where $C>0$, $\varepsilon > 0$ and  $ r $ is the distance function on $ M^n $ from a fixed point $ x_0 \in M $, then $M^n$ is an $n$-dimensional spacelike hyperplane of $\mathbb R^{n+p}_{p}$.
\end{theorem}

\begin{theorem}\label{thm2}
Let $X:M^n\looparrowright\mathbb R^{n+p}_p$ be a spacelike self-shrinker. If the second fundamental form $A$ of $M^n$ is such that $|A|$ converges to zero at infinity, then $M^n$ is an $n$-dimensional spacelike hyperplane of $\mathbb R^{n+p}_p$.
\end{theorem}


\section{Spacelike submanifolds in $\mathbb{R}^{n+p}_ p$ }\label{sec:preliminaries}

We recall that the $(n+p)$-dimensional pseudo-Euclidean space $\mathbb{R}^{n+p}_{p}$ of index $p$ is just the linear space $\mathbb{R}^{n+p}$ endowed with the indefinite flat metric
$$ds^2=\sum_{i=1}^{n}(dx_{i})^{2}-\sum_{\alpha=n+1}^{n+p}(dx_{\alpha})^2.$$

Now, let $X:M^{n}\looparrowright\mathbb{R}^{n+p}_{p}$ be an $n$-dimensional (connected) spacelike submanifold. Using the following convection for indices
\begin{eqnarray*}
1\leq A,B,C,\ldots\leq n+p,\quad 1\leq i,j,k,\ldots\leq n\quad\mbox{and} 
\end{eqnarray*}
$$
n+1\leq \alpha,\beta,\gamma,\ldots\leq n+p,
$$
we choose a local Lorentzian frame field $\{e_i,e_\alpha\}$ in $\mathbb{R}_p^{n+p}$ along the spacelike submanifold $X:M^{n}\looparrowright\mathbb{R}^{n+p}_{p}$ with spacelike frame $\{e_i\}\subset TM$ and timelike frame $\{e_\alpha\}\subset TM^\perp$. Let us also denote by $\{\omega_{B}\}$ the corresponding dual coframe.

The second fundamental form $A$, the curvature tensor $R$ and the normal curvature tensor $R^\perp$ of $M^{n}$ are given by
\begin{eqnarray*}
\omega_{i\alpha}=\sum_{j}h_{ij}^{\alpha}\omega_{j},\quad A=\sum_{i,j,\alpha}h_{ij}^{\alpha}\omega_{i}\otimes\omega_{j}\otimes e_{\alpha},
\end{eqnarray*}
\begin{eqnarray*}
d\omega_{ij}=-\sum_{k}\omega_{ik}\wedge \omega_{kj}-\frac{1}{2}\sum_{k,l}R_{ijkl}\omega_{k}\wedge \omega_{l},
\end{eqnarray*}
\begin{eqnarray*}
d\omega_{\alpha\beta}=-\sum_{\gamma}\omega_{\alpha\gamma}\wedge \omega_{\gamma\alpha}-\frac{1}{2}\sum_{k,l}R_{\alpha\beta kl}^{\perp}\omega_{k}\wedge\omega_{l},
\end{eqnarray*}
where $\{\omega_{BC}\}$ denote the connection $1$-forms on $\mathbb{R}_p^{n+p}$. Moreover, the Gauss equation is given by
\begin{align}\label{equacao de Gauss 1}
R_{ijkl}&=-\sum_{\alpha}(h_{ik}^{\alpha}h_{jl}^{\alpha}-h_{il}^{\alpha}h_{jk}^{\alpha})
\end{align}
and the Ricci equation is
\begin{equation}\label{equacao de Ricci}
R^\perp_{\alpha\beta ij}=\sum_{m}(h^{\alpha}_{mi}h^{\beta}_{mj}-h^{\alpha}_{mj}h^{\beta}_{mi}).
\end{equation}
Denoting by $H^\alpha$ the components of the standard mean curvature vector $\vec{H}$, that
is,
$$
\vec{H} = \sum_\alpha H^\alpha {e}_\alpha = \sum_\alpha \left( \sum_k h^\alpha_{kk} \right) {e}_\alpha,
$$
it is not difficult to verify from \eqref{equacao de Gauss 1} that the components of the Ricci tensor $R_{ik}$ 
satisfy
\begin{equation}\label{tensor Ricc}
R_{ik} = -\sum_\alpha H^\alpha h^\alpha_{ik} + \sum_{\alpha,j} h^\alpha_{ij} h^\alpha_{jk}.
\end{equation}
The covariant derivative of $h^\alpha_{ij}$ is given by
\begin{equation}\label{derivada ricc}
\sum_k h^\alpha_{ijk} \omega_k = dh^\alpha_{ij} + \sum_k h^\alpha_{kj} \omega_{ki} + \sum_k h^\alpha_{ik} \omega_{kj} - \sum_\beta h^\beta_{ij} \omega_{\beta\alpha}. \tag{1.4}
\end{equation}
Hence, it holds the Codazzi equation
\begin{equation}\label{codazzi equation}
h^\alpha_{ijk} = h^\alpha_{ikj}.
\end{equation}
From \eqref{derivada ricc} we obtain the following Ricci formula
\begin{equation}\label{Ricci Formula}
h^\alpha_{ijkl} - h^\alpha_{ijlk} = \sum_m h^\alpha_{mj} R_{mikl} + \sum_m h^\alpha_{im} R_{mjkl} + \sum_\beta h^\beta_{ij} R^\perp_{\alpha\beta kl}.
\end{equation}
Since the Laplacian of $h^\alpha_{ij}$ is defined by $\sum_k h^\alpha_{ijkk}$, from \eqref{equacao de Gauss 1}, \eqref{equacao de Ricci} and \eqref{Ricci Formula} we conclude
$$
\Delta h^\alpha_{ij} = \sum_k h^\alpha_{kkij} - \sum_{k,m,\beta} h^\alpha_{mi} h^\beta_{mj} h^\beta_{kk} - 2 \sum_{k,m,\beta} h^\alpha_{km} h^\beta_{mj} h^\beta_{ik} 
$$
\begin{equation}
+ \sum_{k,m,\beta} h^\alpha_{km} h^\beta_{mk} h^\beta_{ij} + \sum_{k,m,\beta} h^\alpha_{mj} h^\beta_{ki} h^\beta_{mk}. \tag{1.6}
\end{equation}

\section{Main Rigidity Results and Maximum Principles}\label{sec:proof}

Let $(M^n,\langle\,,\rangle)$ be a Riemannian manifold. According to \cite{QH:16}, given a smooth vector field $V$ on $M^n$, the Bakry-Émery-Ricci tensor ${\rm Ric}_V$ as being the following extension of the standard Ricci tensor ${\rm Ric}$
\begin{equation}\label{VRiccTensor}
	\operatorname{Ric}_V := \operatorname{Ric} - \frac{1}{2} \mathcal{L}_V g,
\end{equation}
where $\mathcal{L}_V g$ denotes the Lie derivative of metric tensor $g = \langle, \rangle$. Furthermore, given a function $u \in C^2(M)$, the drift Laplacian operator $\Delta_V$ acting on $u$ is defined by
\begin{equation}\label{Laplacians relation}
	{\Delta}_V u={\rm div}_V(\nabla u)=\Delta u-\langle V,\nabla u \rangle,
\end{equation}
where $\Delta$ is the standard Laplacian of $(M^n,\langle\,,\rangle)$ and the $V-$divergence operator is given by
\begin{equation}\label{div-f}
	{\rm div}_V(W)={\rm div}(W) - \langle V, W \rangle,
\end{equation}
for all tangent vector field $W$ on $M^n$.

In order to prove our first result, we will make use of the following Omori--Yau type maximum principle for the operator $\Delta_V$, which is applicable to spacelike self-shrinkers that are closed with respect to the Euclidean topology, under certain conditions established by H. B. Qiu in \cite{Qiu:2021}.

\begin{lemma}\label{lemaDelta_V}
	Let $X \colon M^m \to \mathbb{R}^{m+p}_p$ be a spacelike self-shrinker, which is closed with respect to the Euclidean topology\footnote{From \cite{JostXin2001}, we know that if $M$ is closed with respect to the Euclidean topology and $o \in M$, the pseudo-distance $z$ is a non-negative proper function on $M$.}. Assume that the origin $o \in M$. If there exists a constant $C > 0$, such that $\|H\| \leq C(d + 1)$, where $d = \langle X, X \rangle$ is the pseudo-distance function. Then for any $f \in C^2(M)$ with 
	$$
	\lim_{x\to\infty} \frac{f(x)}{\log(d(x)+1)} = 0,
	$$
	there exists $\{p_k\} \subset M$, such that
	$$
	\lim_{j\to\infty} f(p_k) = \sup f, \quad 
	\lim_{j\to\infty} |\nabla f|(p_k) = 0, \quad 
	\lim_{j\to\infty} \Delta_V f(p_k) \leq 0.
	$$
\end{lemma}

As an application of Lemma \ref{lemaDelta_V}, we are in position to proof our first rigidity Theorem.

\begin{theorem}\label{TH1}
Let $X:M^n\looparrowright\mathbb R^{n+p}_{p}$ be a spacelike self-shrinker, which is closed with respect to the Euclidean topology and assume that origin $0 \in M$. If there exists a constant $C>0$, such that 
$$||H||\leq C(d+1), \quad  d = \langle X, X \rangle$$ 
where $d$ is the pseudo-distance function, then $M^n$ is an $n$-dimensional spacelike hyperplane of $\mathbb R^{n+p}_{p}$.
\end{theorem}

\begin{proof}
Considering $V=\frac{1}{2}X^\top$ and choosing a local Lorentzian frame field $\{e_i,e_\alpha\}$ in $\mathbb R_p^{n+p}$ along $X:M^n\looparrowright\mathbb R^{n+p}_{p}$ with spacelike $\{e_i\}\subset TM$ and timelike $\{e_\alpha\}\subset TM^\perp$. From \cite[Proposition 2.1]{Xin:2011} we have the following Bochner-Simons type formula
\begin{align}\label{eqaux1}
\Delta |A|^2&=2|\nabla A|^2+  2\langle \nabla_i \nabla_j H, A_{ij} \rangle +2\langle A_{ij} , H  \rangle\langle A_{ik},A_{jk}\rangle+2|R^{\perp}|^{2}-2\sum_{\alpha,\beta}S^{2}_{\alpha\beta},
\end{align}
where $S_{\alpha\beta}=\displaystyle\sum_{i,j}h_{ij}^{\alpha}h_{ij}^{\beta}$ and $A_{ij}=\left(\overline{\nabla}_{e_{i}}e_{j}\right)^{\perp}=\displaystyle\sum_{\alpha}h_{ij}^{\alpha}e_{\alpha}$.
From the self-shrinker equation \eqref{self} it follows that
\begin{equation}\label{nablaei}
\nabla_{e_i} \nabla_{e_j} H = \dfrac{1}{2} A_{ij} - \langle H, A_{ik} \rangle A_{jk} + \dfrac{1}{2} \langle X, e_k \rangle \nabla_{e_i} A_{jk}.
\end{equation}
By adding \eqref{nablaei} in \eqref{eqaux1} and applying Codazzi equation, we conclude the equality
\begin{equation}\label{eq3}
\Delta |A|^2=2|\nabla A|^2+  |A|^2 + \dfrac{1}{2} \langle X, \nabla |A|^2 \rangle +2|R^{\perp}|^{2}-2\sum_{\alpha,\beta}S^{2}_{\alpha\beta},
\end{equation}
Since  $|A|^2$ is the square of norm of the second fundamental form of $M^n$ in $\mathbb{R}^{n+p}_p$, which is nonnegative, we use the same notation for other timelike quantities. Then, from \eqref{Laplacians relation}, \eqref{eq3} we conclude
\begin{equation}\label{eq4}
\Delta_V |A|^2 = |A|^2 + 2|\nabla A|^2 + 2|R^\perp|^2 + 2\sum_{\alpha,\beta}S^{2}_{\alpha\beta}.
\end{equation}
Taking into account that
\begin{equation}
\sum_{\alpha,\beta}S^{2}_{\alpha\beta}\geq \sum_{\alpha}S^{2}_{\alpha\alpha}\geq \frac{1}{p}\left(\sum_{\alpha}S_{\alpha\alpha}\right)^{2} = \frac{1}{p}|A|^{4},
\end{equation}
from \eqref{eq4} we obtain
\begin{equation}\label{ineqaux2}
\Delta_V |A|^{2}\geq\frac{2}{p}|A|^{4}+|A|^{2} = P(A)  |A|^2 \geq 0,
\end{equation}
where $P(A) = \dfrac{2}{p}|A|^2 + 1$ is strictly positive.
On the other hand, under our hypotheses and by Lemma \ref{lemaDelta_V}, there exists a sequence $(p_k) \in M^n$ such that
\begin{equation}\label{lemaemA2}
\lim_{k} |A|^2(p_k) = \sup_M |A|^2, \quad 
\lim_{k} |\nabla |A|^2 |(p_k) = 0, \quad 
\lim_{k} \Delta_V |A|^2(p_k) \leq 0,
\end{equation}
from \eqref{lemaemA2} and \eqref{ineqaux2} we obtain
$$
0 \geq \lim_{k} \Delta_V |A|^2(p_k) \geq \left( \dfrac{2}{p}(\sup |A|^2) + 1 \right) \sup |A|^2 (p_k) \geq 0.
$$
Therefore, we conclude that $|A|^2$ vanishes identically in $M^n$, which means that $M^n$ is totally geodesic and consequently, $M^n$ should be an $n-$dimensional spacelike hyperplane of $\mathbb{R}^{n+p}_p$.

\end{proof}

In order to prove our second result, we will make use of the following Omori--Yau type maximum principle for the operator $\Delta_V$, which is applicable to spacelike self-shrinkers, under certain conditions established by Chen and Qiu in \cite{QH:16}.

\begin{lemma}\label{lema2}
	Let $(M^m, g)$ be a complete Riemannian manifold, $ V $ a $ C^1 $ vector field on $ M $. If $ \operatorname{Ric}_V \geq -F(r)g$, where $ r $ is the distance function on $ M $ from a fixed point $ x_0 \in M $, $ F: \mathbb{R} \to \mathbb{R} $ is a positive continuous function satisfying
	\begin{equation}\label{maximumprinciple2}
	\varphi(t) := \int_{\rho_0 + 1}^t \frac{dr}{\displaystyle\int_{r}^{\rho_0} F(s) \, ds + 1} \to +\infty \quad (t \to +\infty)
	\end{equation}
	for some positive constant $ \rho_0 $. Let $ f \in C^2(M) $ with $ \displaystyle\lim_{x \to \infty} \frac{f(x)}{\varphi(r(x))} = 0 $, then there exist points \( \{x_j\} \subset M \), such that
	\begin{equation}
	\lim_{j \to \infty} f(x_j) = \sup f, \quad \lim_{j \to \infty} |\nabla f|(x_j) = 0, \quad \lim_{j \to \infty} \Delta_V f(x_j) \leq 0.
	\end{equation}
\end{lemma}

As an application of Lemma \ref{lema2}, we obtain our second rigidity result.

\begin{theorem}
	Let $X:M^n\looparrowright\mathbb R^{n+p}_{p}$ be a spacelike self-shrinker with  $||X^\perp||^2 < 2p$. If for any constants $C>0$ and $\varepsilon > 0$ such that the second fundamental form $A$ satisfies
	$$
	|A| \leq \dfrac{C}{r^{1+\varepsilon}}, \quad \text{and} \quad \rho_0 \geq \left( \dfrac{C_0}{\varepsilon} \right)^{1/\varepsilon},  ~~ C_0 = \sqrt{2p}C,
	$$
	where $ r $ is the distance function on $ M $ from a fixed point $ x_0 \in M $, then $M^n$ is an $n$-dimensional spacelike hyperplane of $\mathbb R^{n+p}_{p}$.
\end{theorem}

\begin{proof}
Choosing a local Lorentzian frame field $\{e_i,e_\alpha\}$ in $\mathbb R_p^{n+p}$ along $X:M^n\looparrowright\mathbb R^{n+p}_{p}$ with spacelike $\{e_i\}\subset TM$ and timelike $\{e_\alpha\}\subset TM^\perp$.	Following the same reasoning as in the previous theorem, from \eqref{ineqaux2} follows that 
\begin{equation}\label{1estimativa}
\Delta_V |A|^{2}\geq\frac{2}{p}|A|^{4}+|A|^{2} = P(A)  |A|^2 \geq 0,
\end{equation}	
where $P(A) = \dfrac{2}{p}|A|^2 + 1$ is strictly positive. On the other hand, in order to apply Lemma \ref{lema2}, we compute ${\rm Ric}_V$ to obtain a lower bound for a function $F$ that satisfies the assumptions of the Lemma. Considering $V=-\dfrac{1}{2}X^\top$ From \eqref{VRiccTensor} we get
$$
\begin{aligned}
{\rm Ric}_V(e_i, e_i) 
& = {\rm Ric}(e_i, e_i) - \dfrac{1}{2}(\mathcal{L}_V g)(e_i, e_i) \\
& = - \sum_{\alpha} H^\alpha h^\alpha_{ii} + \sum_{\alpha, j} (h^\alpha_{ij})^2 - \langle \overline{\nabla}_{e_i} V, e_i \rangle \\
& = \dfrac{1}{2}\sum_{\alpha} X_\alpha^\perp h^\alpha_{ii} + \sum_{\alpha, j} (h^\alpha_{ij})^2 + \dfrac{1}{2} \langle \overline{\nabla}_{e_i} X^\top, e_i \rangle \\
& = \dfrac{1}{2}\sum_{\alpha} X_\alpha^\perp h^\alpha_{ii} + \sum_{\alpha, j} (h^\alpha_{ij})^2 + \dfrac{1}{2} \langle \overline{\nabla}_{e_i} (X - X^\perp), e_i \rangle, \\
\end{aligned}
$$
since $\overline{\nabla}_{e_i} = X = e_i$ and $\overline{\nabla}_{e_i} X^\perp = - \displaystyle\sum_{\alpha} X^\perp_\alpha h^{\alpha}_{ii} e_i$, simplifying the terms, we get
\begin{equation}
{\rm Ric}_V(e_i, e_i) = \dfrac{1}{2} +  \sum_{\alpha, j} (h^\alpha_{ij})^2 + \sum_\alpha X^\perp_\alpha h^{\alpha}_{ii}.
\end{equation}
applying the Cauchy inequality and given that $||X^\perp ||^2<2p$, $|A|\leq \dfrac{C}{r^{1+ \varepsilon}}$ for $\varepsilon > 0$ and a positive constant $C>0$ , we conclude that
\begin{equation}
{\rm Ric}_V(e_i, e_i) \geq \dfrac{1}{2} + |A|^2 - \sqrt{2p}|A| \geq - \dfrac{C_0}{r^{1+ \varepsilon}},
\end{equation}
where $C_0 = \sqrt{2p}C$. Choosing $F(r) = \dfrac{C_0}{r^{1+\varepsilon}},$ we have
\begin{equation}
{\rm Ric}_V \geq - F(r) g.
\end{equation}
Now, we will show that $F(r)$ satisfies Lemma \ref{lema2}. In fact, 
$$
\int_r^{\rho_0} F(s)\,ds + 1 
= C_0 \int_r^{\rho_0} \frac{1}{s^{1+\varepsilon}}\,ds + 1 
= \frac{C_0}{\varepsilon} \left( \frac{1}{r^{\varepsilon}} - \frac{1}{\rho_0^{\varepsilon}} \right) + 1,
$$
for $r \longrightarrow + \infty$, we obtain that
$$
\int_r^{\rho_0} F(s)ds + 1 = - \dfrac{C_0}{\varepsilon \rho_0^\varepsilon} + 1 > 0,
$$
since $\rho_0 > \left( \dfrac{C_0}{\varepsilon}  \right)^{1/\varepsilon}$. Thus, 
$$
\varphi(t) := \int_{\rho_0 + 1}^t \frac{dr}{\displaystyle\int_{r}^{\rho_0} F(s) \, ds + 1} = \dfrac{1}{ \left( 1 - \dfrac{C_0}{\varepsilon \rho_0^\varepsilon} \right) } \int_{\rho_0 + 1}^t dr = t  \to +\infty \quad (t \to +\infty).
$$
To conclude, note also that for $f = |A|^2 \in C^2(M)$, we get $
\lim\limits_{r \to \infty}\dfrac{|A|^2}{\varphi(r)} = 0.$ Applying Lemma \ref{lema2} there exist a sequence of point $(p_k) \subset M$, such that 
\begin{equation}\label{aplicacao lema 2}
	\lim_{k \to \infty} |A|^2(p_k) = \sup |A|^2, \quad \lim_{k \to \infty} |\nabla |A|^2|(p_k) = 0, \quad \lim_{k \to \infty} \Delta_V |A|^2(p_k) \leq 0.
\end{equation}
From \eqref{aplicacao lema 2} and \eqref{1estimativa} we obtain
$$
0 \geq \lim_{k} \Delta_V |A|^2(p_k) \geq \left( \dfrac{2}{p}(\sup |A|^2) + 1 \right) \sup |A|^2 (p_k) \geq 0.
$$
Therefore, we conclude that $|A|^2$ vanishes identically in $M^n$, which means that $M^n$ is totally geodesic and consequently, $M^n$ should be an $n-$dimensional spacelike hyperplane of $\mathbb{R}^{n+p}_p$.

\end{proof}

Now, let $(M^n,\langle\,,\rangle)$ be a complete noncompact Riemannian manifold and let $d(\,\cdot\,,o):M^n\rightarrow[0,+\infty)$ denote the Riemannian distance of $M^n$, measured from a fixed point $o\in M^n$. According to \cite[Section 2]{Alias-Caminha-Nascimento:2019}, we say that a smooth function $u\in C^{\infty}(M)$ {\em converges to zero at infinity} when it satisfies the following condition
\begin{equation}\label{distance condition}
	\lim_{d(x,o)\rightarrow+\infty}u(x)=0.
\end{equation}
Keeping in mind this previous concept, in order to prove Theorem~\ref{thm2}, we quote a consequence of item $(a)$ of \cite[Theorem $2.2$]{Alias-Caminha-Nascimento:2019} due Alías, Caminha and Nascimento.

\begin{lemma}\label{lemma Alias-Caminha-Nascimento}
Let $(M^n,\langle\,,\rangle)$ be a complete noncompact Riemannian manifold and let $W\in\mathfrak{X}(M)$ be a smooth vector field on $M^n$. Assume that there exists a nonnegative, non-identically vanishing function $u\in C^{\infty}(M)$ which converges to zero at infinity and such that $\langle\nabla u,W\rangle\geq0$. If ${\rm div}W\geq0$ on $M^n$, then $\langle\nabla u,W\rangle\equiv0$ on $M^n$.
\end{lemma}

As an application of Lemma \ref{lemma Alias-Caminha-Nascimento}, we obtain our final rigidity result.

\begin{theorem}
	Let $X:M^n\looparrowright\mathbb R^{n+p}_p$ be a spacelike self-shrinker. If the second fundamental form $A$ of $M^n$ is such that $|A|$ converges to zero at infinity, then $M^n$ is an $n$-dimensional spacelike hyperplane of $\mathbb R^{n+p}_p$.
\end{theorem}

\begin{proof}
Let us suppose by contradiction that $M^n$ is not a spacelike hyperplane of $\mathbb R^{n+p}_{p}$ or, equivalently, that $|A|$ does not vanish identically on $M^n$. By an argument similar to the preceding theorem, the sctrict positivity of $P(A)$ in \eqref{ineqaux2}, we get that $\Delta_V |A|^2 \geq 0$. Thus, taking the smooth vector field $W=\nabla|A|^2$, from \eqref{div-f}, \eqref{Laplacians relation} and \eqref{ineqaux2} follows that
\begin{equation}\label{div X}
0\leq \Delta_V|A|^2 =  {\rm div}_V (\nabla |A|^2) = {\rm div}( \nabla |A|^2 ) - \langle V, \nabla |A|^2 \rangle,
\end{equation}
for $V=\nabla |A|^2$, we conclude that
\begin{equation}\label{divthm2}
{\rm div} W = \Delta_V |A|^2 + |\nabla |A|^2 |^2 \geq 0.
\end{equation}
Moreover, choosing the smooth function $u=|A|^2$ we also have that
\begin{equation}\label{ineqaux1C}
\langle\nabla u,W\rangle=|\nabla|A|^2|^2\geq0.
\end{equation}
Consequently, since we are also supposing that $|A|$ converges to zero at infinity, we can apply Lemma~\ref{lemma Alias-Caminha-Nascimento} to get that
$\langle\nabla u,W\rangle\equiv0$ on $M^n$. So, returning to \eqref{ineqaux1C} we conclude that $|A|$ must be constant on $M^n$. Therefore, from \eqref{distance condition} we have that $|A|$ is identically zero on $M^n$ and, hence, we reach at a contradiction. Therefore, $M^n$ must indeed be a spacelike hyperplane of $\mathbb{R}^{n+p}_p$.
\end{proof}





\end{document}